\documentclass[reqno,11pt]{amsart}
\usepackage{amsmath,amssymb,color}
\usepackage{amsfonts, amscd, epsfig, amsmath, amssymb,enumerate}
\usepackage{graphicx}
\usepackage{graphics}
\usepackage{color}
\usepackage{mathrsfs}
\usepackage{todonotes}
\usepackage{soul}

 \usepackage[usenames,dvipsnames]{pstricks}
 \usepackage{pstricks-add}
 \usepackage{epsfig}
 \usepackage{pst-grad} 
 \usepackage{pst-plot} 
 \usepackage[space]{grffile} 
 \usepackage{etoolbox} 
\usepackage[margin=3cm]{geometry}
 \makeatletter 
 \patchcmd\Gread@eps{\@inputcheck#1 }{\@inputcheck"#1"\relax}{}{}
 \makeatother

\newtheorem{theorem}{Theorem}[section]
\newtheorem{lemma}[theorem]{Lemma}

\newtheorem{remark}[theorem]{Remark}

\usepackage{tikz}
\usetikzlibrary{backgrounds}
\usetikzlibrary{patterns,fadings}
\usetikzlibrary{arrows,decorations.pathmorphing}
\usetikzlibrary{decorations}
\usetikzlibrary{calc}
\usetikzlibrary{shapes.misc}

\usepackage{setspace}

\definecolor{light-gray}{gray}{0.95}
\usepackage{float}
\usepackage[colorlinks=true,linkcolor=blue,citecolor=magenta]{hyperref}
\def\centerarc[#1](#2)(#3:#4:#5){\draw[#1] ($(#2)+({#5*cos(#3)},{#5*sin(#3)})$) arc (#3:#4:#5);}

\newcommand{\vertiii}[1]{{\left\vert\kern-0.25ex\left\vert\kern-0.25ex\left\vert #1 
    \right\vert\kern-0.25ex\right\vert\kern-0.25ex\right\vert}}

\numberwithin{equation}{section}
\numberwithin{figure}{section}

\newcommand{\bb}[1]{{\mathbb #1}}

\newcommand{\<}{\big\langle}
\renewcommand{\>}{\big\rangle}

\renewcommand{\epsilon}{\varepsilon}

\newcommand{\R}{\mathbb R}

\renewcommand{\P}{\mathbb P}

\newcommand{\E}{\mathbb E}

\newcommand{\En}[1]{\lim_{N \rightarrow \infty}  \bb{E}^N_{\mu_N} \Big[ \big| \int_0^t #1 ds \big|\Big]}

\allowdisplaybreaks 

\begin{document}
	
\title{Long-time behavior of SSEP with slow boundary}

\author{Linjie Zhao}
\address{Inria, Univ. Lille, CNRS, UMR 8524 - Laboratoire Paul Painlev\'e, F-59000 Lille}
\email{linjie.zhao@inria.fr}


\begin{abstract}
We consider the symmetric simple exclusion process with slow boundary first introduced in [Baldasso {\it et al.}, Journal of Statistical Physics, 167(5), 2017]. We prove a law of large number for the empirical measure of the process under a longer time scaling instead of the usual diffusive time scaling.
\end{abstract}

\maketitle

\section{Introduction}

Interacting particle systems in contact with reservoirs have been investigated in various literature  \cite{de2011current,Erignoux18,eyink1990hydrodynamics}.  We study in this article the symmetric simple exclusion process (SSEP) on a line segment $\{1,\ldots,N-1\}$  with slow boundary, which is first introduced  by Baldasso {\it et al.} \cite{baldasso2017exclusion}.   Here, $N$ is the scaling parameter. There is at most one particle per site.  In the bulk, a particle jumps to one of its neighbors at rate one provided the target site is empty. Fix parameters $c > 0, \theta \geq 0$ and $\alpha, \, \beta \in(0,1)$.
At the boundary site $1$ (resp. $N-1$), a particle is created at rate $c \alpha N^{-\theta}$ (resp. $c \beta N^{-\theta}$)  if site  $1$ (resp. $N-1$) is empty, and a particle is destroyed at rate $c (1-\alpha) N^{-\theta}$ (resp. $c(1- \beta) N^{-\theta}$)  if site  $1$ (resp. $N-1$) is occupied.  Therefore, the particle  density of the left (resp. right) reservoir  is $\alpha$ (resp. $\beta$), and the interaction strength between the bulk and the reservoirs is $cN^{-\theta}$. The hydrodynamic equation of the model turns out to be the heat equation with Dirichlet boundary conditions if $\theta < 1$, with Robin boundary conditions if $\theta = 1$ and with Neumann boundary conditions if $\theta > 1$. We refer the readers to \cite{baldasso2017exclusion} for more background of the model. 

The hydrodynamic limit of the model is considered under the diffusive time scaling, i.e., with time speeded up by $N^2$ and space divided by $N$. The aim of this article is to consider the behavior of the process under a longer time scaling $N^{2+\gamma},\,\gamma > 0$.  Since the process is irreducible,  it has a unique invariant measure. Under the invariant measure, the empirical measure converges in probability  to the stationary solution of the corresponding hydrodynamic equations as $N \rightarrow \infty$. This is called hydrostatic limit  \cite{baldasso2017exclusion,tsunoda2020hydrostatic}.  The hydrostatic limit could be formally interpreted as taking  $\gamma = \infty$.  For $0 < \gamma < \infty$, it is natural to expect that the  limit of the empirical measure should coincide with the hydrostatic limit. This is indeed true for $\theta \leq 1$, since the stationary solution of the corresponding hydrodynamic equation is unique in this case. For $\theta > 1$, since the stationary solution is not unique, three regimes appear depending on whether $\gamma < \theta - 1$, $\gamma = \theta - 1$ or $\gamma > \theta - 1$.  See Theorem \ref{thm:longtime} for details.

Despite the simple structure of the model, it has attracted a lot of attention since then.  The equilibrium/non-equilibrium fluctuations from the hydrodynamic limit are considered in \cite{franco2019non,gonccalves2020non}.  The large deviation of the SSEP with slow boundary are investigated in \cite{bouley2021steady,derrida2021large,franco2021large}.

The paper is organized as follows. In Section \ref{sec:notation} we define the model rigorously via its infinitesimal generator, review the hydrodynamic/hydrostatic limit already proven in \cite{baldasso2017exclusion,tsunoda2020hydrostatic}, and state the main result of the article. In Section \ref{sec:preliminary} we introduce the notation of Dirichlet forms and prove the so-called replacement lemmas under a longer time scaling. The estimates involving the Dirichlet forms are mostly borrowed from \cite{baldasso2017exclusion}. The proof of Theorem \ref{thm:longtime} is presented in Section \ref{sec:proof}.

\section{Notation and Results}\label{sec:notation}

\subsection{The model.}  The state space of the process $(\eta_t)_{t \geq 0}$ is $\Omega_N:= \{0,1\}^{I_N}$, where $I_N := \{1,\ldots,N-1\}$ . Here, $N$ is the scaling parameter. For a configuration $\eta \in \Omega_N$, $\eta (x) = 1$ if and only if there is a particle at site $x$. Fix parameters $c > 0$, $\theta \geq 0$ and $\alpha, \beta \in (0,1)$.  The parameter $\theta$ denotes the strength of  interaction with reservoirs and $\alpha,\,\beta$ are the particle densities of reservoirs.   The generator  $L_N$  of the process $(\eta_t)_{t \geq 0}$ is given by
\[L_{N}=L_{N, 0}+L_{N, b}^{\alpha}+L_{N, b}^{\beta}.\]
Above,  the generator $L_{N,0}$ of the bulk dynamics acting on functions $f:\Omega_N \rightarrow \R$ is 
\[	\left(L_{N, 0} f\right)(\eta)=\sum_{x=1}^{N-2}\left[f\left(\eta^{x, x+1}\right)-f(\eta)\right],\]
where $\eta^{x,y}$ is the configuration obtained from $\eta$ by exchanging the values of $\eta(x)$ and $\eta (y)$, i.e., $\eta^{x,y}(x) = \eta (y)$, $\eta^{x,y} (y) = \eta (x)$ and $\eta^{x,y} (z) = \eta (z)$ for $z \neq x,\,y$.
The generators $L_{N, b}^{\alpha}$ and $L_{N, b}^{\beta}$  correspond to the boundary effects, and are given by
\begin{align*}
	(L_{N, b}^{\alpha} f)(\eta):=c N^{-\theta} r_{\alpha}(\eta)\left[f\left(\eta^{1}\right)-f(\eta)\right], \quad 
	(L_{N, b}^{\beta} f)(\eta):=c N^{-\theta} r_{\beta}(\eta)\left[f\left(\eta^{N-1}\right)-f(\eta)\right],
\end{align*}
where  $\eta^x$ is the configuration obtained from $\eta$ by flipping the value of $\eta(x)$, i.e., $\eta^x (x) = 1 - \eta (x)$ and $\eta^x (z) = \eta (z)$ for $z \neq x$, and
\[r_{\alpha}(\eta)=\alpha(1-\eta(1))+(1-\alpha) \eta(1),\quad r_{\beta}(\eta)=\beta(1-\eta(N-1))+(1-\beta) \eta(N-1).\]

Denote by $\mu_N$ the initial measure of the process. For any positive integer $k$, let $C^k [0,1]$ be the family of functions on $[0,1]$ such that the $m$-th derivative is uniformly continuous in $(0,1)$ for any $m \leq k$.

\subsection{Hydrodynamic limit.}  It has been proven in \cite{baldasso2017exclusion} that phase transitions occur for the SSEP with slow boundary, depending on whether $\theta < 1$, $\theta = 1$ or $\theta > 1$.  To state the hydrodynamic limit, we impose the following assumptions on the initial measure $\mu_N$: there exists a measurable initial density profile $\rho_0: [0,1] \rightarrow [0,1]$ such that  for any $G \in C[0,1]$, 
\[\lim_{N \rightarrow \infty} \frac{1}{N} \sum_{x=1}^{N-1}  \eta (x) G \big(\tfrac{x}{N}\big) = \int_0^1 \rho_0 (u) G(u)\,du \]
 in probability with respect to $\mu_N$.  The following result characterize the macroscopic density profile under the diffusive time scaling.
 
 \begin{theorem}[Cf. {\cite[Theorem 2.8]{baldasso2017exclusion}}]\label{thm:dynamic}
For any $t \geq 0$ and  for any $G \in C[0,1]$, 
\[\lim_{N \rightarrow \infty} \frac{1}{N} \sum_{x=1}^{N-1}  \eta_{tN^2} (x) G \big(\tfrac{x}{N}\big) = \int_0^1 \rho (t,u) G(u)\,du \]
in probability, where

\noindent $(i)$ if $0 \leq \theta < 1$, then $\rho (t,u)$ is the unique weak solution to the heat equation with Dirichlet boundary
\begin{equation}\label{dirichlet}
	\left\{\begin{array}{ll}\partial_{t} \rho(t,u)=\Delta \rho (t,u), & u \in(0,1), \quad t \geq 0 \\ \rho(t,0)=\alpha, \rho(t,1)=\beta, & t \geq 0, \\ \rho (0,u)=\rho_0 (u), & u \in[0,1]\end{array}\right.
\end{equation}

\noindent $(ii)$ if $ \theta = 1$, then $\rho (t,u)$ is the unique weak solution to the heat equation with Robin boundary
\begin{equation}\label{robin}
	\left\{\begin{array}{ll}\partial  \rho(t,u)=\Delta \rho(t,u), & u \in(0,1), \quad t \geq 0 \\ \partial_{u} \rho(t,0)=c\left(\rho(t,0)-\alpha\right), & t \geq 0, \\ \partial_{u} \rho(t,1)=c\left(\beta-\rho(t,1)\right), & t \geq 0 \\ \rho(0,u)=\rho_0 (u), & u \in[0,1]\end{array}\right.
\end{equation}

\noindent $(iii)$ if $ \theta > 1$, then $\rho (t,u)$ is the unique weak solution to the heat equation with Neumann boundary
\begin{equation}\label{neumann}
	\left\{\begin{array}{ll}\partial_{t} \rho(t,u)=\Delta \rho(t,u), & u \in(0,1), \quad t \geq 0 \\ \partial_{u} \rho(t,0)=0, \partial_{u} \rho(t,1)=0, & t \geq 0 \\ \rho(0,u)=\rho_0(u), & u \in[0,1]\end{array}\right.
\end{equation}
 \end{theorem}

We refer the readers to \cite{baldasso2017exclusion} for rigorous definitions of weak solutions the above PDEs. 

\subsection{Hydrostatic limit.} Since the process $(\eta_t)_{t \geq 0}$  is irreducible, it has a unique invariant measure denoted by $\mu_N^{ss}$.  The following result characterize the macroscopic density profile under the invariant measure $\mu_N^{ss}$.

\begin{theorem}[Cf. {\cite[Theorem 2.2]{baldasso2017exclusion} and \cite{tsunoda2020hydrostatic}}]\label{thm:static}
 For any $G \in C[0,1]$, 
\[\lim_{N \rightarrow \infty} \frac{1}{N} \sum_{x=1}^{N-1}  \eta(x) G \big(\tfrac{x}{N}\big) = \int_0^1 \bar{\rho }(u) G(u)\,du \]
in probability with respect to $\mu_N^{ss}$, where
\begin{equation*}
	\bar{\rho}(u)=\left\{\begin{array}{ll}(\beta-\alpha) u+\alpha, & \text { if } \theta \in[0,1) \\ \frac{c(\beta-\alpha)}{2+c} u+\alpha+\frac{\beta-\alpha}{2+c}, & \text { if } \theta=1 \\ \frac{\beta+\alpha}{2}, & \text { if } \theta \in(1, \infty)\end{array}\right.
\end{equation*}
\end{theorem}

We say $\rho:[0,1] \rightarrow [0,1]$ is a stationary solution to \eqref{dirichlet} if
\[\Delta \rho (u) = 0, \;u \in (0,1),\quad\rho (0) = \alpha, \quad\rho(1) = \beta.\]
Stationary solutions to \eqref{robin} and \eqref{neumann} could be defined in the same way. Note that the hydrostatic limits $\bar{\rho}$ are stationary solutions to the corresponding hydrodynamic equations as stated in Theorem \ref{thm:dynamic}.  We underline that the stationary solution to \eqref{dirichlet} and \eqref{robin} is unique, while the stationary solution to \eqref{neumann} is not unique. Indeed, any constant function is a stationary solution to \eqref{neumann}. The above theorem tells us that the correct choice is $(\alpha + \beta)/2$.

\subsection{Long-time limit.} In this subsection, we state the main result of the article. Fix $\gamma > 0$.  We shall consider the process speeded up by $N^{2+\gamma}$. Denote by $(\eta^N_t)_{t \geq 0}$ the process with generator $N^{2+\gamma} L_N$. Then $(\eta^N_t)_{t \geq 0}$ and $(\eta_{tN^{2+\gamma}})_{t \geq 0}$ have the same distribution. We are interested in the long time behavior of the empirical measure $\pi^N_t$ of the process defined as
\[\pi^N_t (du) = \frac{1}{N} \sum_{x=1}^{N-1} \eta^N_t (x) \delta_{x/N} (du),\]
where $\delta_{x/N} (du)$ is the Dirac measure on the point $x/N$. Whence, $\pi^N_t$ is a random measure on $[0,1]$ with total mass bounded by one. With this notation, for any $G \in C[0,1]$,
\[\<\pi^N_t,G\> =  \frac{1}{N} \sum_{x=1}^{N-1} \eta^N_t (x)  G \big(\tfrac{x}{N}\big).\]

 For $\theta > 1$ and $0 < \gamma \leq \theta - 1$, we assume  that the average number of particles converges in the following sense: there exists $m_0 \in [0,1]$ such that
 \begin{equation}\label{initial}
\lim_{N \rightarrow \infty} E_{\mu_N} \Big[\Big|\frac{1}{N-1} \sum_{x=1}^{N-1} \eta(x) - m_0\Big|\Big] = 0.
 \end{equation}
We underline that we impose no restrictions on the initial measure $\mu_N$ in the rest of the cases. 

 Denote by $\P^N_{\mu_N}$ the probability measure on $D([0,\infty),\Omega_N)$ associated to the process $(\eta^N_t)_{t \geq 0}$ and the initial measure $\mu_N$, and by $\E^N_{\mu_N}$ the corresponding expectation. 

We now state  the law of large numbers for the empirical measure $\pi^N_t$. 

\begin{theorem}\label{thm:longtime}
For any $t > 0$ and for any $G \in C[0,1]$,
\[\En{\Big\{\<\pi^N_s,G\> - \int_0^1 \rho_{\theta,\gamma} (s,u) G(u) du \Big\}} = 0,\]
where 
\begin{equation}\label{stationrysol}
	\rho_{\theta,\gamma}(t,u)=\left\{\begin{array}{ll}(\beta-\alpha) u+\alpha, & \text { if } 0 \leq \theta<1 \\ 
		\frac{c(\beta-\alpha)}{2+c} u+\alpha+\frac{\beta-\alpha}{2+c}, & \text { if } \theta=1,\\ 
	  m_0, & \text { if } \theta>1,\,0 < \gamma < \theta - 1,\\ 
		\frac{\beta+\alpha}{2}, & \text { if } \theta>1,\,\gamma > \theta - 1,\\
		\frac{\beta+\alpha}{2} + \Big(m_0 - \frac{\alpha+\beta}{2}\Big)e^{-2ct} , & \text { if } \theta>1,\,\gamma = \theta - 1.
	\end{array}\right.
\end{equation} 
\end{theorem}

\begin{remark}
	Note that $\rho_{\theta,\infty} = \bar{\rho}$. This is not surprising since the hydrostatic limit stated in Theorem \ref{thm:static} could be formally interpreted as taking $\gamma = \infty$.  Compared with Theorem \ref{thm:static}, the above theorem states that phase transition occurs even in the supercritical case  $\theta > 1$.
\end{remark}

\begin{remark}
The result should also hold if the density reservoirs vary slowly with time, i.e., if replacing $\alpha$ (resp. $\beta$) with some smooth function $\rho_- (t): \R_+ \rightarrow (0,1)$ (resp. $\rho_+ (t): \R_+ \rightarrow (0,1)$). This is called  quasi-static hydrodynamic limit\;\cite{de2021quasi,de2015quasi}. 
\end{remark} 

\section{Preliminary results}\label{sec:preliminary}

In this section, we introduce the notion of Dirichlet forms and prove several replacement lemmas in different regimes, which are crucial in  the proof of Theorem \ref{thm:longtime}.

\subsection{Dirichlet form.} For a probability measure $\mu$ on $\Omega_N$ and a function $g: \Omega_N \rightarrow \R$, the Dirichlet forms corresponding to the bulk/boundary dynamics are defined as
\begin{align*}
D_{N, 0}(g, \mu)&:=\frac{1}{2} \sum_{x=1}^{N-2} \sum_{\eta \in \Omega_{N}}\left(g\left(\eta^{x, x+1}\right)-g(\eta)\right)^{2} \mu(\eta),\\ 
D_{N, b}^{\alpha}(g, \mu)&:=\frac{1}{2} \sum_{\eta \in \Omega_{N}} c N^{-\theta} r_{\alpha}(\eta)\left(g\left(\eta^{1}\right)-g(\eta)\right)^{2} \mu(\eta),\\
D_{N, b}^{\beta}(g, \mu)&:=\frac{1}{2} \sum_{\eta \in \Omega_{N}} c N^{-\theta} r_{\beta}(\eta)\left(g\left(\eta^{N-1}\right)-g(\eta)\right)^{2} \mu(\eta).
\end{align*}
For any two functions $f,g: \Omega_N \rightarrow \R$, denote
\[\<f,g\>_\mu = \sum_{\eta \in \Omega_{N}} f(\eta) g (\eta) \mu(\eta).\]
For any density profile $\lambda: [0,1] \rightarrow [0,1]$, let $\nu^N_{\lambda(\cdot)}$ be the product measure on $\Omega_N$ with marginals given by 
\[\nu_{\lambda(\cdot)}^{N} \{\eta:\eta(x)=1\}=\lambda \left(\tfrac{x}{N}\right), \quad x \in I_N.\]
In particular, if $\lambda (\cdot) \equiv \rho$ for some $\rho \in [0,1]$, we simply write $\nu^N_\rho$. 

The following lemma compares $\<L_N g,g\>_\mu$ with $D_N (g,\mu)$.

\begin{lemma}\label{lem:dirichlet} $(i)$ Let $\lambda: [0,1] \rightarrow (0,1)$ be a smooth density profile such that there exists a neighborhood of $0$ where $\gamma (\cdot)= \alpha$, and a neighborhood of $1$ where $\gamma (\cdot)= \beta$. Let $f$ be a $\nu^N_{\lambda(\cdot)}$-density, 
\[f \geq 0, \quad \sum_{\eta \in \Omega_{N}} f(\eta) \nu^N_{\lambda(\cdot)} (\eta) = 1.\] 
Then
	\[\<L_N \sqrt{f},\sqrt{f}\>_{\nu^N_{\lambda(\cdot)}} = -(1/2) D_{N,0} (\sqrt{f}, \nu^N_{\lambda(\cdot)}) - D_{N,b}^\alpha (\sqrt{f}, \nu^N_{\lambda(\cdot)}) - D_{N,b}^\beta (\sqrt{f}, \nu^N_{\lambda(\cdot)}) + \mathcal{O} (N^{-1}),\]
where $|\mathcal{O} (N^{-1})| \leq C N^{-1}$ for some finite constant $C$.

\noindent $(ii)$ Let $\rho \in (0,1)$ be a constant. Let $f$ be a $\nu^N_{\rho}$-density. Then
\[\<L_N \sqrt{f},\sqrt{f}\>_{\nu^N_{\rho}} = - D_{N,0} (\sqrt{f}, \nu^N_{\rho}) - D_{N,b}^\alpha (\sqrt{f}, \nu^N_{\rho}) - D_{N,b}^\beta (\sqrt{f}, \nu^N_{\rho}) + \mathcal{O} (N^{-\theta}).\]
\end{lemma}

The first statement $(i)$ is a direct consequence of \cite[Lemma 5.1 $(ii)$ and Lemma 5.2]{baldasso2017exclusion}.  The second statement $(ii)$ follows directly from \cite[Lemma 5.1 $(i)$ and Corrollary 5.3]{baldasso2017exclusion}. For this reason, we omit the proof here.

The following lemma bound the occupation variables at the boundary sites by the corresponding Dirichlet forms.

\begin{lemma}\label{lem:boundary}
Let $\lambda: [0,1] \rightarrow (0,1)$ be a smooth density profile such that there exists a neighborhood of $0$ where $\gamma (\cdot)= \alpha$, and a neighborhood of $1$ where $\gamma (\cdot)= \beta$. Let $f$ be a $\nu^N_{\lambda(\cdot)}$-density. Then there exists a finite constant $C_\alpha$ such that for any $B > 0$,
\[\big| \sum_{\eta \in \Omega_{N}} (\eta (1) - \alpha) f (\eta) \nu^N_{\lambda(\cdot)}(\eta) \big| \leq C_\alpha  B + C_\alpha  N^{\theta} B^{-1} D^\alpha_{N,b} (\sqrt{f}, \nu^N_{\lambda(\cdot)}).\]
The same result holds at the right boundary: there exists a finite constant $C_\beta$ such that for any $B > 0$,
\[\big| \sum_{\eta \in \Omega_{N}} (\eta (N-1) - \beta) f (\eta) \nu^N_{\lambda(\cdot)}(\eta) \big| \leq C_\beta  B + C_\beta N^{\theta} B^{-1} D^\beta_{N,b} (\sqrt{f}, \nu^N_{\lambda(\cdot)}).\]
\end{lemma}

The above lemma is a direct consequence of \cite[Lemmas 5.6 and 5.7]{baldasso2017exclusion}. For that reason, we omit the proof here.

\subsection{Replacement lemmas.} In this subsection, we prove several replacement lemmas under the longer time scaling $N^{2+\gamma}$, $\gamma > 0$.  

The following lemma states that in the subcritical  regime $0 \leq \theta < 1$, we could replace the occupation variable at the boundary sites with the corresponding particle density of reservoirs.

\begin{lemma}[Replacement lemma for the case $0 \leq \theta < 1$.]\label{lem:replace_sub}
	Suppose $0 \leq \theta < 1$. Then for any $t > 0$,
	\begin{equation*}
		\lim_{N \rightarrow \infty} \bb{E}^N_{\mu_N} \Big[\big| \int_0^t (\eta_{s}^N (1) - \alpha) \,ds\big| \Big] = 0.
	\end{equation*}
	The same result holds with $\eta_{s}^N(1)$ replaced with $\eta_{s}^N(N-1)$ and $\alpha$ with $\beta$.
\end{lemma}

\begin{proof} Let $\lambda: [0,1] \rightarrow (0,1)$ be a smooth density profile such that there exists a neighborhood of $0$ where $\gamma (\cdot)= \alpha$, and a neighborhood of $1$ where $\gamma (\cdot)= \beta$.  By the relative entropy inequality (cf. \cite[A.1.8]{klscaling}),  for any $A > 0$, the expectation in the lemma could be bounded from above by 
\begin{equation}\label{eqn3}
	\frac{H(\mu_N|\nu^N_{\lambda(\cdot)})}{AN} + \frac{1}{AN} \log \bb{E}^N_{\nu^N_{\lambda(\cdot)}} \Big[ \exp \Big\{AN \big| \int_0^t (\eta_{s}^N (1) - \alpha) \,ds \big|  \Big\} \Big],
\end{equation}
where for any probability  measures $\mu,\,\nu$ on $\Omega_N$ such that $\mu$ is absolutely continuous with respect to $\nu$, $H(\mu|\nu)$ is the relative entropy of $\mu$ with respect to $\nu$ defined as
\[H(\mu|\nu) = \sum_{\eta \in \Omega_{N}} \mu(\eta) \log \frac{\mu(\eta)}{\nu(\eta)}.\]
It is not hard to prove that $H(\mu_N|\nu^N_{\lambda(\cdot)}) \leq C_\lambda N$ for some finite constant $C_\lambda$.  Therefore, the first term in \eqref{eqn3} is bounded by $C_\lambda/A$. In the sequel, we shall take $A = A(N) \rightarrow \infty$ as $N \rightarrow \infty$. Whence the first term in \eqref{eqn3} vanishes in the limit.  Since $$\lim_{N \rightarrow \infty} r_N^{-1} \log (a_N + b_N) = \max \{\lim_{N \rightarrow \infty} r_N^{-1} \log a_N, \lim_{N \rightarrow \infty} r_N^{-1} \log b_N\}$$
for any positive sequences $\{a_N\}_{N \geq 1},\,\{b_N\}_{N \geq 1}$ and $\{r_N\}_{N \geq 1}$ such that $\lim_{N \rightarrow \infty} r_N = \infty$, we could remove the modulus inside the exponential for the second term in \eqref{eqn3}.  By the Feynman-Kac formula (cf. \cite[Lemma A.1.7.2]{klscaling}), the second term in \eqref{eqn3} is bounded by
\begin{equation}\label{eqn4}
t \sup_{f\, \text{density}} \Big\{  \sum_{\eta \in \Omega_{N}} (\eta (1) - \alpha) f (\eta) \nu^N_{\lambda(\cdot)}(\eta) + \frac{N^{1+\gamma}}{A} \<L_N \sqrt{f},\sqrt{f}\>_{\nu^N_{\lambda(\cdot)}} \Big\}.
\end{equation}
By Lemma \ref{lem:dirichlet} $(i)$,
\[\<L_N \sqrt{f},\sqrt{f}\>_{\nu^N_{\lambda(\cdot)}} \leq - D_{N,b}^\alpha (\sqrt{f}, \nu^N_{\lambda(\cdot)}) + \mathcal{O} (N^{-1}).\] 
Together with Lemma \ref{lem:boundary}, for any $B > 0$, we may bound \eqref{eqn4} by
\[t \sup_{f:\, \text{$\nu^N_{\lambda(\cdot)}$-density}} \Big\{  C_\alpha  B + C_\alpha  N^{\theta} B^{-1} D^\alpha_{N,b} (\sqrt{f}, \nu^N_{\lambda(\cdot)}) - \frac{N^{1+\gamma}}{A} D_{N,b}^\alpha (\sqrt{f}, \nu^N_{\lambda(\cdot)}) + \mathcal{O} (N^{\gamma}/A)\Big\}\]
for some finite constant $C_\alpha$. Taking $B = C_\alpha A N^{\theta - 1 - \gamma}$ and $A = N^\gamma \log N$, the above term is bounded by $t \big( C_\alpha^2 N^{\theta-1} \log N + \mathcal{O} (1/ \log N) \big)$, which converges to zero as $N \rightarrow \infty$ since $\theta < 1$. This concludes the proof.
\end{proof}

Let $m^N (\eta)$ be the average number of particles in the system 
\[m^N (\eta) = \frac{1}{N-1} \sum_{x=1}^{N-1} \eta (x).\]
Denote $m^N_t = m^N (\eta^N_t)$. The next result states that in the supercritical regime $\theta > 1$, we could replace the occupation variables at the boundary sites with the average number of particles in the system.

\begin{lemma}[Replacement lemma for the case $ \theta > 1$.]\label{lem:replace_sup} 
Suppose $\theta > 1$. Then for any $t > 0$,
\begin{align*}
	\lim_{N \rightarrow \infty}\, \bb{E}^N_{\mu_N} \Big[\big| \int_0^t (\eta_{s}^N(1) - m^N_{s}) \,ds\big| \Big] = 0.
\end{align*}
The same result holds with $\eta_{s}^N(1)$ replaced with $\eta_{s}^N(N-1)$ 
\end{lemma}

\begin{proof} The proof is similar to that of Lemma \ref{lem:replace_sub}, and we only sketch the proof here. Fix a constant $\rho \in (0,1)$.  By the relative entropy inequality (cf. \cite[A.1.8]{klscaling}),  for any $A > 0$, the expectation in the lemma could be bounded from above by 
\begin{equation}\label{eqn5}
	\frac{H(\mu_N|\nu^N_\rho)}{AN} + \frac{1}{AN} \log \bb{E}^N_{\nu^N_\rho} \Big[ \exp \Big\{ AN \big| \int_0^t (\eta_{s}^N(1) - m^N_{s}) \,ds\big| \Big\} \Big].
\end{equation}
As in Lemma \ref{lem:replace_sub}, the first term is bounded by $C/A$ for some finite constant $C$.   By the Feynman-Kac formula (cf. \cite[Lemma A.1.7.2]{klscaling}), the second term in \eqref{eqn5} is bounded by
\begin{equation}\label{eqn6}
t \sup_{f:\, \nu^N_\rho\text{-density}} \Big\{ \sum_{\eta \in \Omega_{N}} (\eta(1) - m^N (\eta)) f(\eta) \nu^N_\rho (\eta) + \frac{N^{1+\gamma}}{A} \<L_N \sqrt{f},\sqrt{f}\>_{\nu^N_{\rho}}  \Big\}.
\end{equation}
We may rewrite $\eta(1) - m^N (\eta)$ as a telescope sum
\[\frac{1}{N-1} \sum_{x=1}^{N-1} \sum_{y=1}^{x-1} (\eta(y) - \eta(y+1)).\]
Making the change of variables $\eta \mapsto \eta^{y,y+1}$,
\[\sum_{\eta \in \Omega_{N}} (\eta(1) - m^N (\eta)) f(\eta) \nu^N_\rho (\eta)  = \frac{1}{2(N-1)} \sum_{x=1}^{N-1} \sum_{y=1}^{x-1} \sum_{\eta \in \Omega_{N}}   (\eta(y) - \eta(y+1)) \big(f(\eta)- f(\eta^{y,y+1}) \big) \nu^N_\rho (\eta). \] 
By Cauchy-Schwarz inequality, for any $B > 0$, we may bound the last term by
\begin{align*}
&\frac{B}{4(N-1)} \sum_{x=1}^{N-1} \sum_{y=1}^{x-1} \sum_{\eta \in \Omega_{N}} \big(\sqrt{f}(\eta)- \sqrt{f}(\eta^{y,y+1}) \big)^2 \nu^N_\rho (\eta) \\
&\qquad \qquad + \frac{1}{4B(N-1)} \sum_{x=1}^{N-1} \sum_{y=1}^{x-1} \sum_{\eta \in \Omega_{N}} \big(\sqrt{f}(\eta)+ \sqrt{f}(\eta^{y,y+1}) \big)^2 \nu^N_\rho (\eta) \\
&\leq \frac{B}{2} D_{N,0} (\sqrt{f},\nu^N_\rho) + \frac{N}{B}.
\end{align*}
The last inequality follows from the basic inequality $(a+b)^2 \leq 2 (a^2 + b^2)$ and the fact that $f$ is a density with respect to $\nu^N_\rho$.  By Lemma \ref{lem:dirichlet} $(ii)$, 
\[\<L_N \sqrt{f},\sqrt{f}\>_{\nu^N_{\rho}} \leq  - D_{N,0} (\sqrt{f}, \nu^N_{\rho}) + \mathcal{O} (N^{-\theta}).\]
Whence, \eqref{eqn6} is bounded by
\[t \sup_{f:\, \nu^N_\rho\text{-density}} \Big\{  \frac{B}{2} D_{N,0} (\sqrt{f},\nu^N_\rho) + \frac{N}{B} - \frac{N^{1+\gamma}}{A} D_{N,0} (\sqrt{f}, \nu^N_{\rho}) + \mathcal{O} (N^{1+\gamma -\theta} /A).\Big\}\]
Taking $B = 2N^{1+\gamma}/A$ and $A = N^\gamma / (\log N)$, the above term is bounded by $1/ (2 \log N) + \mathcal{O} (N^{1-\theta} \log N)$, which converges to zero as $N \rightarrow \infty$ since $\theta > 1$. This concludes the proof.
\end{proof}

The next lemma concerns about the long time behavior of the average number of particles in the supercritical case.

\begin{lemma}[Replacement lemma for the average particle number.]\label{lem:replace_average}  Suppose $\theta > 1$.  Recall $m_0$ defined in \eqref{initial} is the average number of particles at the initial time. For any $t > 0$, 
\begin{enumerate}[(i)]
	\item if $0 \leq \gamma < \theta -1$, then
	\begin{equation*}
		\lim_{N \rightarrow \infty}\bb{E}^N_{\mu_N} \Big[\big| \int_0^t \big(m^N_{s} - m_0\big) \,ds\big| \Big] = 0,
	\end{equation*}
	
	\item if $\gamma =  \theta-1$, then
		\begin{equation*}
		\lim_{N \rightarrow \infty}\bb{E}^N_{\mu_N} \Big[\big| \int_0^t \big(m^N_{s} - m_s\big) \,ds\big| \Big] = 0,
	\end{equation*}
where \[m_s = \frac{\alpha+\beta}{2} + \Big(m_0 - \frac{\alpha+\beta}{2}\Big)e^{-2cs}.\]
	
	\item  if $\gamma >  \theta-1$, then
		\begin{equation*}
		\lim_{N \rightarrow \infty}\bb{E}^N_{\mu_N} \Big[\big| \int_0^t \big(m^N_{s} - (\alpha+\beta)/2\big) \,ds\big| \Big] = 0.
	\end{equation*}
\end{enumerate}
\end{lemma}

\begin{proof}
The  statement $(ii)$ is a direct consequence of  \cite[Proposition 4.5]{tsunoda2020hydrostatic}. For the rest of the statements, consider the martingale $\mathfrak{m}^N_t$	defined as 
\begin{equation}\label{eqn2}
\mathfrak{m}^N_t := m^N_{t} -  m^N_{0} -  \int_0^t N^{2+\gamma} L_N m^N_{s }\,ds,
\end{equation}
whose quadratic variation at time $t$ is given by
\[\int_0^t \big\{N^{2+\gamma} L_N (m^N_{s })^2 - 2  m^N_s N^{2+\gamma} L_N m^N_s  \big\} ds.\]
A simple calculation shows that the quadratic variation of $\mathfrak{m}^N_t$ is bounded by $C N^{\gamma - \theta}$ for some finite constant $C$, and that the integral term in \eqref{eqn2} equals
\[\frac{c N^{2+\gamma-\theta}}{N-1} \int_0^t \big( \alpha - \eta_{s}^N(1) + \beta -  \eta_{s}^N(N-1) \big)\,ds.\]

If $0 \leq \gamma < \theta -1$,  by Doob's inequality, for any $T > 0$,
\[\lim_{N \rightarrow \infty} \E^N_{\mu_N} \Big[ \sup_{0 \leq t \leq T} \big(\mathfrak{m}^N_t\big)^2\Big] = 0. \]
The integral term  in \eqref{eqn2} is of order $N^{1+\gamma - \theta}$, which converges to zero as $N \rightarrow \infty$ uniformly in a bounded time interval. Therefore,
\[\lim_{N \rightarrow \infty} \bb{E}^N_{\mu_N} \big[\sup_{0 \leq t \leq T} |m^N_{t}-m_0| \big] = 0\] 
This proves the first statement $(i)$. 
	
If $\gamma > \theta - 1$, divided  by $N^{1+\gamma - \theta}$ in \eqref{eqn2}, we have
\[\lim_{N \rightarrow \infty} \E^N_{\mu_N} \Big[ \sup_{0 \leq t \leq T} \big(N^{\theta - \gamma - 1}\mathfrak{m}^N_t\big)^2\Big] = 0. \]
Since $m^N_t \leq 1$, by \eqref{eqn2}, 
\begin{equation*}
	\lim_{N \rightarrow \infty}\bb{E}^N_{\mu_N} \Big[ \big|\int_0^t \big( \alpha - \eta_{s}^N (1) + \beta -  \eta_{s}^N(N-1) \big)\,ds\big|\Big] = 0.
\end{equation*}
By Lemma \ref{lem:replace_sup}, we could replace $\eta_{s}^N(1)$ and $\eta_{s}^N (N-1)$ in the time integral with $m^N_{s}$. This concludes the proof.
\end{proof}

\section{Proof of Theorem \ref{thm:longtime}}\label{sec:proof}

In this section, we prove Theorem \ref{thm:longtime} depending on whether $0 \leq \theta < 1$, $\theta = 1$ or $\theta > 1$. For $H \in C^2 [0,1]$, consider the martingale defined as
\begin{equation}\label{matgale1}
	\begin{aligned}
		M^N_t (H) = \<\pi^N_t,H\> - \<\pi^N_0,H\> 
		-  \int_0^t N^{2+\gamma}  L_N \<\pi^N_s,H\>  ds,
	\end{aligned}
\end{equation}
whose quadratic variation at time $t$ is given by
\[ \int_0^t \big\{N^{2+\gamma}  L_N \<\pi^N_s,H\>^2 - 2\<\pi^N_s,H\> N^{2+\gamma}  L_N \<\pi^N_s,H\>\big\} ds.\]
Direct calculations show that the quadratic variation of $M^N_t (H)$ is  bounded by $C_H (N^{\gamma - 1} + N^{\gamma - \theta})$ for some finite constant $C_H$. Therefore, for any $T > 0$,
\[\lim_{N \rightarrow \infty} \E^N_{\mu_N} \Big[ \sup_{0 \leq t \leq T} \big(N^{-\gamma} M^N_t (H)\big)^2 \Big] = 0. \]
Since there is at most one particle per site, $|\<\pi^N_t,H\>| \leq ||H||_\infty$ uniformly in $t$, where $||H||_\infty := \max_{u \in [0,1]} |H(u)|$ is the uniform norm. Divided by $N^{\gamma}$ in \eqref{matgale1}, 
\[ \En{N^{2}  L_N \<\pi^N_s,H\>} = 0.
\]
Direct calculations yield that
\begin{equation*}
\begin{aligned}
	&N^{2}  L_N \<\pi^N_s,H\> =\<\pi^N_s,H^{\prime \prime}\>-\eta_{s}^N (N-1)  H^\prime (1)+\eta_{s}^N(1) H^\prime (0) \\
	&+ c N^{1-\theta}\left(\alpha-\eta_{s}^N(1)\right) H(0)+c N^{1-\theta}\left(\beta-\eta_{s}^N(N-1)\right) H(1) + \mathcal{O}(N^{-\theta} + N^{-1}).
\end{aligned}
\end{equation*}
Whence,
\begin{equation}\label{eqn1}
\begin{aligned}
\lim_{N \rightarrow \infty}  \bb{E}^N_{\mu_N} \Big[& \big| \int_0^t \Big\{ \<\pi^N_s,H^{\prime \prime}\>-\eta_{s}^N (N-1)  H^\prime (1)+\eta_{s}^N(1) H^\prime (0) \\
&+ c N^{1-\theta}\left(\alpha-\eta_{s}^N(1)\right) H(0)+c N^{1-\theta}\left(\beta-\eta_{s}^N(N-1)\right) H(1) \Big\} ds \big|\Big] = 0.
\end{aligned}
\end{equation}

\medspace

\subsection{The case $0 \leq \theta < 1$.} In this subsection, we prove Theorem \ref{thm:longtime} for the case $0 \leq \theta < 1$. Fix $G \in C[0,1]$. The main technique here is to find an appropriate function $H$ such that $H^{\prime \prime} = G$ on (0,1) and that $H(0) = H(1) = 0$.  With such a function $H$,  the second line in \eqref{eqn1} vanishes and the result follows by the corresponding replacement lemmas.

\begin{proof}[Proof of Theorem \ref{thm:longtime} in the case $0 \leq \theta < 1$.]
	For $G \in C[0,1]$, let 
	\begin{equation}\label{eqn:H}
		H (u) = H_G (u) = \int_0^u \int_0^v G(w) \,dw\,dv + u \int_0^1 (v-1) G(v)\,dv.
	\end{equation}
	It is easy to check that $H \in C^2 [0,1]$, $H^{\prime \prime} = G$ on $(0,1)$, and  that
	\begin{align*}
		H(0) = H(1) = 0,\quad  H^\prime (0) = \int_0^1 (u-1) G(u)\,du, \quad H^\prime (1) = \int_0^1 u G(u)\,du.
	\end{align*}
	Substituting the function $H$  into \eqref{eqn1}, 
	\begin{equation*}
		\En{  \Big\{\<\pi^N_s,G\>
			- \int_0^1 \big[ (\eta^N_{s}(N-1) - \eta^N_{s}(1))u + \eta^N_{s}(1)\big] G(u)\,du\Big\}} = 0.
	\end{equation*}
By Lemma \ref{lem:replace_sub}, we could replace $\eta^N_s (1)$ (resp. $\eta^N_s (N-1)$) with $\alpha$ (resp. $\beta$). This concludes the proof for the case $0 \leq \theta < 1$.
\end{proof}

\subsection{The case $\theta = 1$.} In this subsection, we prove Theorem \ref{thm:longtime} for the case $\theta = 1$.  Fix $G \in C[0,1]$. We need to find an appropriate function $H$ such that $H^{\prime \prime} = G$ on (0,1) and that the coefficients of $\eta (1)$ and $\eta(N-1)$ vanish in \eqref{eqn1}.  Note that in this case we do not need any replacement lemma.

\begin{proof}[Proof of Theorem \ref{thm:longtime} in the case $\theta = 1$.]
For $G \in C[0,1]$, let 
\begin{align*}
H(u) =    \int_0^1 \Big(\frac{v}{2+c}  - \frac{1+c}{c(2+c)} \Big)G(v)dv + u \int_0^1 \Big(\frac{cv}{2+c} -  \frac{1+c}{2+c}\Big) G(v)dv 
+ \int_0^u \int_0^v G(w) \,dw\,dv. 
\end{align*}
It is easy to check that $H \in C^2 [0,1]$, $H^{\prime \prime} = G$ on $(0,1)$, and  that
\begin{align*}
	H(0) =  \int_0^1 \Big(\frac{v}{2+c}  - \frac{1+c}{c(2+c)} \Big)G(v)dv, &\quad
	H(1) = \int_0^1 \Big(-\frac{v}{2+c}  - \frac{1}{c(2+c)} \Big)G(v)dv,\\
	H^\prime (0) = \int_0^1 \Big(\frac{cv}{2+c} -  \frac{1+c}{2+c}\Big) G(v)dv, &\quad
	H^\prime (1) =   \int_0^1 \Big(\frac{cv}{2+c} +  \frac{1}{2+c}\Big) G(v)dv.
\end{align*}
Taking the function $H$ into \eqref{eqn1} and calculating the coefficients of $\eta(1)$ and $\eta (N-1)$, we have 
\begin{align*}
- H^\prime (1) - c H (1) = 0, \quad H^\prime (0) - c H(0) = 0.
\end{align*}
The constant term  in \eqref{eqn1}  is given by
\[c \alpha H(0) + c \beta H(1) = \int_0^1 \Big(\frac{c(\alpha - \beta)}{2+c}v - \alpha  -  \frac{\beta - \alpha}{2+c}\Big) G(v)dv. \]
Whence, the integrand in \eqref{eqn1} is equal to
\[\<\pi^N_s,G\> - \int_0^1 \Big(\frac{c(\beta-\alpha)}{2+c}v + \alpha  +  \frac{\beta - \alpha}{2+c}\Big) G(v)dv. \]
By \eqref{eqn1}, the time integral of the above term converges in $L^1(\P^N_{\mu_N})$ to zero as $N \rightarrow \infty$. This concludes the proof for the case $\theta =1$.
\end{proof}

\subsection{The case $\theta > 1$.} In this subsection, we prove Theorem \ref{thm:longtime} for the case $\theta > 1$.  In this case, the last line in \eqref{eqn1} converges to zero as $N \rightarrow \infty$. Whence, we do not need special properties of the function $H$.

\begin{proof}[Proof of Theorem \ref{thm:longtime} in the case $\theta > 1$.]
For $G \in C[0,1]$, let 
	\begin{align*}
		H(u) =   \int_0^u \int_0^v G(w) \,dw\,dv. 
	\end{align*}
Taking the function $H$ into \eqref{eqn1} and by Lemma \ref{lem:replace_sup},
\[\En{\Big\{\<\pi^N_s,G\> - m^N_s (H^\prime (1) - H^\prime (0))\Big\}} = 0.\]	
It is easy to check
\[H^\prime (1) - H^\prime (0) = \int_0^1 G(u) du.\]
Therefore,
\[\En{\Big\{\<\pi^N_s,G\> - m^N_s \int_0^1 G(u) du\Big\}} = 0.\]
By Lemma \ref{lem:replace_average}, we conclude the proof for the case $\theta >1$.
\end{proof}

\medspace

\textbf{Acknowledgments.}  
Zhao thanks the financial support from the ANR grant MICMOV (ANR-19-CE40-0012) of the French National Research Agency (ANR).

\bibliographystyle{plain}
\bibliography{zhaoreference.bib}
\end{document}